\documentclass[12pt,letterpaper, reqno]{amsart}
\usepackage{amssymb,latexsym, amsmath, amsxtra}
\usepackage[dvips]{graphics}

\theoremstyle{plain}
\newtheorem{theorem}{Theorem}[section]
\newtheorem{corollary}[theorem]{Corollary}

\newtheorem{proposition}[theorem]{Proposition}
\theoremstyle{definition}

\theoremstyle{remark}

\numberwithin{equation}{section}
\numberwithin{theorem}{section}
\numberwithin{table}{section}
\numberwithin{figure}{section}

\def\({\left(}
\def\){\right)}

\begin{document}
\title{Palindromic random trigonometric polynomials}
\author{J. Brian Conrey, David W. Farmer, and \"Ozlem Imamo\-glu}

\thanks{
Research of the first two authors supported by the
American Institute of Mathematics
and the National Science Foundation}

\thispagestyle{empty}
\vspace{.5cm}
\begin{abstract}
We show that if a real trigonometric polynomial has few
real roots, then the trigonometric polynomial obtained by writing the
coefficients in reverse order must have many real roots.  This is used
to show that a class of random
trigonometric polynomials has, on average, many real roots.
In the case that the coefficients of a 
 real trigonometric polynomial are independently and identically
distributed, but with no other assumptions on the distribution,
the expected fraction of real zeros is at
least one-half.  This result is best possible.
\end{abstract}

\address{
{\parskip 0pt
American Institute of Mathematics\endgraf
conrey@aimath.org\endgraf
farmer@aimath.org\endgraf
\null
ETH Zurich\endgraf
ozlem@math.ethz.ch
}
  }

\maketitle

\section{Introduction}

A random trigonometric polynomial is a function of the form
\begin{equation}\label{eqn:rtp}
F(x)=\sum_{k=0}^{K} a_k \cos(k x)  
\end{equation}
where the $a_k$ 
are random variables. 

We give a simple proof of the following very general result:

\begin{theorem}\label{thm:iid}
Suppose $C(x)$ is a real trigonometric polynomial
\begin{equation}
C(x)=\sum_{k=L}^{M} a_k \cos(k x)
\end{equation}
where the $a_k$ are independently and identically distributed
real random variables.
Then, for those $C(x)$ which are not identically zero,
the expected number of real zeros in each period is
at least~$L+M$.
\end{theorem}
In particular, on average at least half the zeros of $C(x)$ are real. 

The surprising thing about the theorem is that we obtain a fairly
strong conclusion with absolutely no assumption on the 
coefficients beyond the fact that they are iid.  We now contrast this
with previous results on this topic.

In the case that the $a_k$ 
are normally distributed,
the Kac-Rice formula~\cite{EK} provides a quick way to determine
the expected number of real zeros of~$F$.  Here ``number of real zeros''
refers to the number of real zeros in one period. In the case that
the normal distributions also are independent and have equal variance,
Dunnage~\cite{D} showed that for large $K$ the expected number of real zeros
is asymptotically $\frac{2 K}{\sqrt{3}}$.  That is, about 
$57.7\%$ of the zeros are real.

There are some exact results for the expected number of real zeros of~\eqref{eqn:rtp}
when the coefficients are not Gaussian.  For example, \cite{J}
considered the case that the distribution of coefficients is in the
domain of attraction of a stable law with exponent $0<\alpha\le 2$.
More generally, \cite{ST}
established a result assuming that 
the expected value of $|a_k|^3$ is finite.
Both of those results have an additional condition that the
probability $a_k=0$ is zero.

There are related results for random algebraic polynomials.
The papers~\cite{IM1, IM2} consider the case where $a_k$ is
in the domain of attraction
of the normal law, and the
probability $a_k=0$ is zero.
And in~\cite{IZ} the requirement is that
$a_k$ have bounded density
(which implies that the probability $a_k=0$ is zero)
and the expected value of $|a_k|^7$ is finite.

The need to assume that the probability $a_k=0$
is zero arises because if
there is a positive probability of the polynomial being
identically zero then the expected number of real zeros is
meaningless.   Thus, the identically zero polynomials
must be eliminated before computing the expected value,
and available methods cannot do that.

Restrictions on the moments
of the $a_k$ arise for several reasons.  Typically upper bounds 
on the moments are needed 
for technical reasons, and it is possible that existing results
can be improved.  
Lower bounds on the variance, as in~\cite{J, SR, ST}, are unavoidable.
This is required to rule out, for example, the case that
every $a_k$ is identically equal to~1, for which
all the roots of~\eqref{eqn:rtp} are real.
This shows that the condition of a lower bound on the variance, 
as in~\cite{J, SR, ST},
cannot be eliminated.

Despite the need for conditions on $a_k$ in order to determine the 
exact expected
number of real zeros of~\eqref{eqn:rtp}, 
we give an easy proof of Theorem~\ref{thm:iid}, which has
no conditions on the distribution.  We avoid the technical issue
of dealing with the possibility that the polynomial is identically
zero because our method allows us to ignore that case.  The cost is
that we obtain a lower bound instead of an exact expected value.
Nevertheless, our result is best possible because if $a_k$
is very likely to be zero with a small chance of equaling~1, for example,
then the nonzero~$C(x)$ usually have only one term and our lower
bound is obtained.

To state a more general version of the above theorem, we will
refer to
a joint probability distribution on 
$a_k$, $0\le k\le K$, 
as \emph{palindromic} if  
the $a_k$ are independent and
$(a_0,\ldots,a_K)$ has the same distribution as $(a_K,\ldots,a_0)$.
Note that if $a_0,\ldots,a_K$ are iid then the joint distribution
is palindromic.

\begin{theorem}\label{thm:palindrome}
If the joint probability distribution of $a_0,\ldots,a_K$ is
palindromic, then the expected number of real zeros of
\begin{equation}
C(x)=\sum_{k=0}^{K} a_k \cos(k x),
\end{equation}
for those $C(x)$ which are not identically zero,
is at least~$K$.
\end{theorem}

To obtain Theorem~\ref{thm:iid} set
$a_0,\ldots,a_{L-1}$ and $a_{K-L+1},\ldots,a_{K}$ to always be zero,
and let the other coefficients be independently and identically
distributed.

Theorem~\ref{thm:palindrome} follows from a simple observation
about trigonometric
polynomials which doesn't actually have anything to do with
randomness.  
We introduce a small amount of notation and
give the main proof in the next section.  In Section~\ref{sec:other}
we discuss other cases in which these methods apply, such as
odd trigonometric polynomials.

\section{Cosine polynomials}

Throughout this section $a_k$ is either a real number or a real
random variable.
If $a=(a_0,\ldots,a_K)$ define
\begin{equation}
C_a(x)=\sum_{k=0}^K a_k \cos(k x)
\end{equation}
and let $\overline{a}=(a_K,\ldots,a_0)$.  Finally, if $f$ is
a trigonometric polynomial let $Z(f)$ denote the number of real zeros
of~$f$ in one period, counted with multiplicity.

All our results on random polynomials follow from the following
completely deterministic statement.

\begin{proposition}\label{prop:main}
If $a=(a_0,a_1,\ldots,a_K)\not = (0,0,\ldots,0)$ then
\begin{equation}\label{eqn:palindrome}
Z(C_{a}) + Z(C_{\overline{a}}) \ge 2 K.
\end{equation}
\end{proposition}

That is, an even trigonometric polynomial and its ``reverse'' have
on average at least half of their zeros real.

The above Proposition is a slight generalization and strengthening of
Theorem~2 of~\cite{BEFL}.  Our result does not assume that any
particular $a_k$ is nonzero (except that $Z(0)$ is not defined),
and a slightly
larger lower bound than stated in~\cite{BEFL} follows by setting
$a_0=a_K=0$. For the case of $\{0,1\}$ polynomials
in Theorem~3 of~\cite{BEFL}, we obtain a lower bound of $n+1$,
which improves their lower bound of~$n/4$.

\begin{proof}[Proof of Theorem~\ref{thm:palindrome}]
Take the expected value of \eqref{eqn:palindrome}.
By assumption $Z(C_{a})$ and $Z(C_{\overline{a}})$ have the
same distribution, and therefore the same expected value.
\end{proof}

Note that the expected value of $Z(C_{a})$ exists because it is
the average of a bounded measurable function.

As another application of Proposition~\ref{prop:main} we have
the following, which appears to be nontrivial if~$L\ge 2$.

\begin{corollary}
If $a=(a_0,\ldots,a_K)$ with $a_k=0$ for $k\le L$, then $Z(C_a)\ge 2 L$.
\end{corollary}

\begin{proof} 
If $C_a\equiv 0$ then the conclusion is true.  Otherwise,
$C_{\overline{a}}$ has (trigonometric) degree
at most $K-L$ and so can have at most~$2(K-L)$ zeros.  In other words,
$Z(C_{\overline{a}})\le 2K-2L$.
\end{proof}

\subsection{Proof of Proposition~\ref{prop:main}}
Real trigonometric polynomials are just another
view of self-reciprocal polynomials, and our proof involves
counting and estimating, in two different ways,
the number of zeros on the unit circle of a self-reciprocal polynomial.

Let
\begin{equation}
f(z)=\sum_{k=0}^K a_k z^k
\end{equation}
and set
\begin{equation}
h(z)=z^K  f(1/z) + z^{-K}f(z) .
\end{equation}
Note that 
\begin{equation}
h(e^{i\theta}) = 2\sum_{k=0}^K a_k \cos((K-k)\theta),
\end{equation}
so real zeros of $C_{\overline{a}}$ correspond to zeros of $h(z)$
on $|z|=1$.

Since  $h(z)$ is real on $|z|=1$, the number of zeros  on $|z|=1$
is at least
as large as the number of sign changes of $h(e^{2\pi i j/2K})$
for $j=0,\ldots,2K$.  
Here a ``sign change'' of $\lambda_j=h(e^{2\pi i j/2K})$ means
that $\lambda_j \lambda_{j+1}\le 0$.

We have
\begin{align}
h(e^{2\pi i j/2K})=\mathstrut& (-1)^j ( f(e^{-2\pi i j/2K}) + f(e^{2\pi i j/2K}))\cr
=\mathstrut& 2 (-1)^j  \sum_{k=0}^K a_k \cos(2 \pi j/2 K) \cr
=\mathstrut& 2 (-1)^j  C_{a}(2 \pi j/2 K).
\end{align}
Putting everything together:
\begin{align}
Z(C_{\overline{a}}) \ge\mathstrut& \#\{\text{sign changes of } h(e^{2\pi i j/2K}) \} \cr
=\mathstrut& \#\{\text{sign\ changes of } (-1)^j C_{a}({2\pi  j/2K}) \cr
\ge\mathstrut& 2K - \#\{\text{sign changes of } C_{a}({2\pi  j/2K}) \} \cr
\ge\mathstrut& 2K - Z(C_{a}),
\end{align}
as claimed.


\section{Sine polynomials and other cases}\label{sec:other}

The proof of Proposition~\ref{prop:main} is easily modified to handle
other cases, such as odd trigonometric polynomials:
$\sum b_k \sin(k x)$.  For that case merely replace
$h(z)$ by
\begin{equation}
h^-(z)=z^K  f(1/z) - z^{-K}f(z) .
\end{equation}

More generally, fix $\varphi$
and define 
\begin{equation}
h^\varphi(z)=e^{i \varphi} z^K  f(1/z) -e^{-i \varphi} z^{-K}f(z) .
\end{equation}
Following the proof of Proposition~\ref{prop:main} we find that
if $X = \tan(\varphi)$ is any real number (or $\infty$, suitably interpreted) then
\begin{equation}\label{eqn:genpalindrome}
Z(F_{X,a}) + Z(F_{X,\overline{a}}) \ge 2 K,
\end{equation}
where
\begin{equation}
F_{X,a}(x) = \sum_{k=0}^K a_k \cos(k x)+ X \ \sum_{k=0}^K a_k \sin(k x).
\end{equation}

\end{document}